\documentclass[11pt]{article}
\usepackage[applemac]{inputenc}
\usepackage{amssymb}
\usepackage{amsmath}
\usepackage{graphics}
\usepackage{latexsym}
\usepackage{ae,aecompl}
\usepackage{amsthm}
\usepackage{mathrsfs}
\usepackage{amsfonts}
\usepackage[pdftex]{graphicx}
\usepackage{enumerate}
\usepackage[usenames]{color}
\usepackage[colorlinks=true]{hyperref}

\newtheorem{theorem}{Theorem}[]
\newtheorem{remark}{Remark}[]
\newtheorem{definition}{Definition}[]
\newtheorem{proposition}[theorem]{Proposition}
\newtheorem{lemma}[theorem]{Lemma}

\newcommand{\comp}{\raisebox{0.1ex}{\scriptsize $\circ$}}
\newcommand{\eqdist}{ \overset{(d)}{=} }

\def\osc{\operatorname{osc}}

\newcommand{\keywords}[1]{ \noindent {\footnotesize
             {\small \em Keywords and phrases.} {\sc #1} } }
\newcommand{\ams}[2]{  \noindent {\footnotesize
             {\small \em AMS {\rm 2000} subject classifications.
             {\rm Primary {\sc #1}; secondary {\sc #2}} } } }

\title{\Huge{Iterating Brownian motions, \emph{ad libitum}}}
\author{Nicolas Curien and Takis Konstantopoulos}
\date{December 2011}
\begin{document}
\maketitle

\abstract{Let $B_1,B_2, ...$ be independent one-dimensional Brownian motions defined over the whole real line such that $B_i(0)=0$ for every $i \geq 1$. We consider the $n$th iterated Brownian motion $W_n(t)= B_n(B_{n-1}(...(B_2(B_1(t)))...))$. Although the sequences of processes $(W_{n})_{n \geq 1}$ do not converge in a functional sense, we prove that the finite-dimensional marginals converge. As a consequence, we deduce that the random occupation measures of $W_n$ converge towards a random probability measure $\mu_\infty$. We then prove that $\mu_\infty$ almost surely has a continuous density which must be thought of as the local time process of the infinite iteration of independent Brownian motions.}

\vspace*{2mm}

\vspace*{2mm}
\ams{60J65,60J05}{60G57,60E99}

\section{Introduction}  \label{introduction}
Let $B_+=(B_+(t), t \ge 0)$ and $B_-=(B_-(t), t \ge 0)$ be independent standard one-dimensional Brownian motions starting from $0$.
The process 
$B(t) := B_+(t)$ if $t \geq 0$ and $B(t):=B_-(-t)$ if $t \le 0$ is called a two-sided Brownian motion. In this paper we  study the iterations of independent (two-sided) Brownian motions. Formally,  let $B_1,B_2, ...$ be a sequence of independent two-sided Brownian motions and,
for every $n \geq 1$ and every $t \in \mathbb{R}$,
set   
\begin{eqnarray*}W_n(t)&:=& B_n\Bigg(B_{n-1}\bigg(...\Big(B_2\big(B_1(t)\big)\Big)...\bigg)\Bigg).  \end{eqnarray*}  Burdzy \cite{Bur93} studied sample path properties of the random function 
$ t \mapsto W_2(t)$ 
and coined the terminology of (second) ``iterated Brownian motion'' for this object.
A motivation for this study is that iterated Brownian motion can be used to
construct a solution to the 4-th order
PDE $\partial u/\partial t = \tfrac{1}{8} \partial^4u /\partial x^4$; see \cite{Fun79}. This model has triggered a lot of work, see \cite{Ber96b,Bur93,Bur95,Bur95,ES99,KL99} and the references therein. 

\begin{figure}[!h]
 \begin{center}
 \includegraphics[width=4cm]{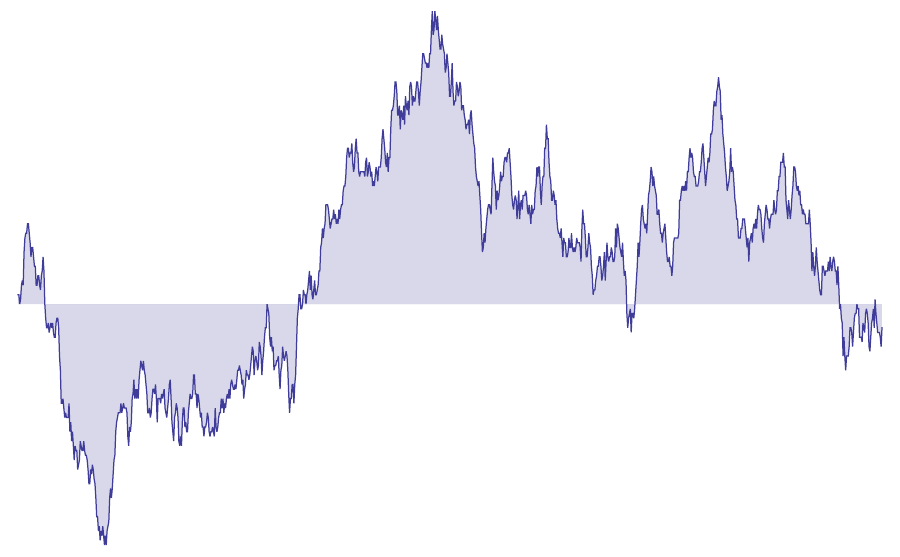}
  \includegraphics[width=4cm]{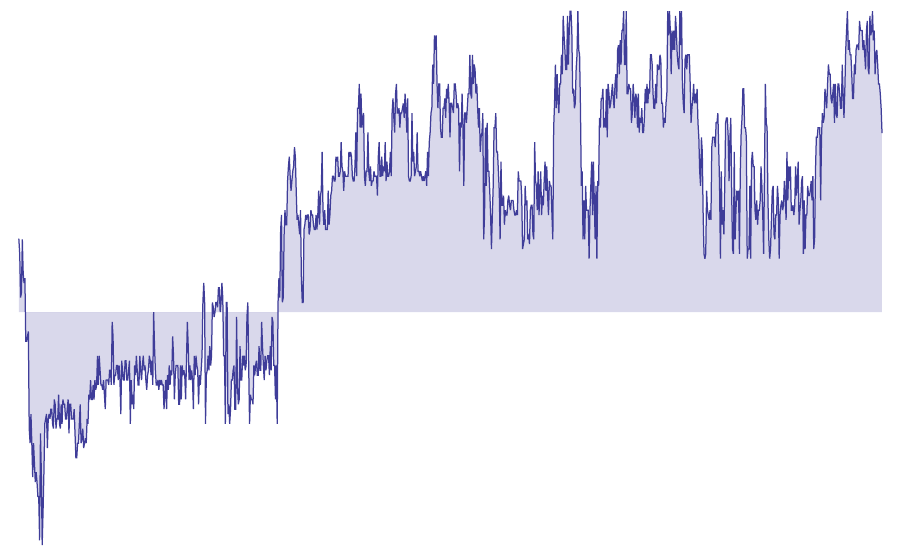}
   \includegraphics[width=4cm]{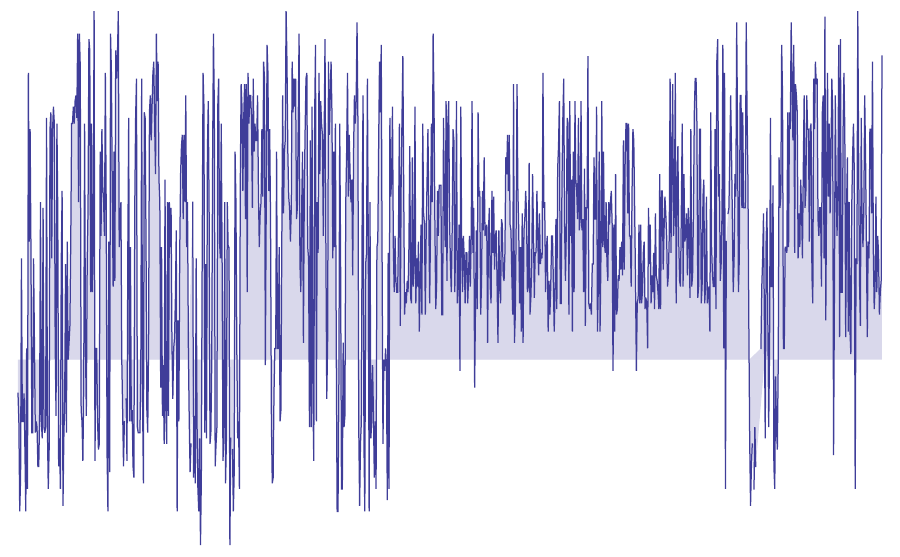}
 \caption{ \label{wild} First, second and third iterations of Brownian motions.
}
 \end{center}
 \end{figure}

Another motivation is that the process $W_n$
 is not
a semimartingale (unless $n=1$). Indeed, for $n=2$, a simple calculation shows
that the quadratic variation of $W_2=B_2 \comp B_1$ does not
exist, but its quartic variation does. Similarly the $2^{n}$--variation of $W_n$ is finite for $n \geq 1$. Hence, as $n$ increases, the process $W_n$ becomes wilder and wilder. 
Also, $W_n$ is self-similar
with index $2^{-n}$, i.e.,
\[
(W_n(\alpha t), t \in \mathbb R) \eqdist (\alpha^{2^{-n}}
W_n(t), t \in \mathbb R),
\]
for all $\alpha >0$.  See Fig.\,\ref{wild} for a comparison of
sample paths of $W_n$, $n=1,2,3$. 
All these reasons make one suspect that  convergence of the laws of the $W_n$'s 
in a ``nice'' function space 
(e.g., the space of continuous functions) is impossible
(see Remark \ref{nonconv}).

However, following the principle of Berman saying that wild functions 
must have smooth local times (see the nice survey \cite{GH80}), 
we prove that the occupation measures $\mu_n$ of $W_n$ over $[0,1]$ 
converge in distribution towards a random  measure $\mu_\infty$ 
which can be interpreted as the occupation measure of the infinite 
iteration ``$W_\infty$". 
More precisely, let $\mu_n$ be the random probability measure defined by  
\begin{equation}	\label{mun}
\int_{ \mathbb{R}} \mathrm{d}\mu_n\ f = \int_0^1 \mathrm{d}t\, 
f\big(W_n(t)\big),  
\end{equation}
for every Borel $f : \mathbb{R} \to \mathbb{R}_+$. 
Our main result is a limit theorem for the sequence 
$(\mu_n)_{n \geq 1}$.
Restricting the integration to the unit interval 
is convenient and poses no loss of generality due 
to the self-similarity property of $W_n$.

Let $ \mathcal{M}$ be the set of all positive Radon measures on $\mathbb{R}$.
Although we focus on the real line, the interested reader should consult
\cite{Kal76} for the general theory of random measures.
Endow $\mathcal{M}$ with the topology 
$\mathcal{ T}$ of vague convergence, that is, the weakest topology 
which makes the mappings 
\[
\mu \in   \mathcal{ M} \mapsto \mu f:= \int_{\mathbb R} \mathrm{d}\mu f ,
\quad  f \in C_K,
\]
continuous. (Here, $C_K$ is the set of continuous functions with compact support.)
A \emph{random measure} is a random element of the space
$( \mathcal{M}, \mathcal{T})$, viewed as a measurable space with
$\sigma$-algebra generated by the sets in $\mathcal{T}$. 
A sequence $\lambda_{1}, \lambda_{2}, \ldots$ of random measures 
converges in distribution towards a random measure $\lambda$ 
if for any bounded continuous mapping $F : ( \mathcal{M}, \mathcal{T}) 
\to \mathbb{R}$ we have $E[F(\lambda_{i})] \to E[F(\lambda)]$ as $i \to \infty$.
We write  \begin{eqnarray*} \lambda_n &\xrightarrow[]{(d)}& \lambda,  \end{eqnarray*} to denote this notion.
Convergence of $\lambda_n$ to $\lambda$ in distribution
is equivalent to:
$\lambda_{n} f \xrightarrow[]{(d)} \lambda f$,
for any continuous $f \in C_K$ (see Theorem 4.2 in \cite{Kal76}). The
latter convergence is convergence in distribution of real-valued 
random variables.

\begin{theorem}  \label{main} 
There exists a random measure
$\mu_\infty$ such that $\mu_n \xrightarrow[]{(d)} \mu_\infty$. 
Moreover $\mu_\infty(\mathbb R)=1$, a.s.
The random probability measure $\mu_\infty$ almost surely admits a density 
$(L_a)_{a \in \mathbb{R}}$ with respect to the Lebesgue measure such 
that $a \mapsto L_a$ has compact support and is H\"older continuous 
with exponent $1/2- \varepsilon$ for every $ \varepsilon>0$. 
\end{theorem}

We can think of $\mu_\infty$ as the occupation measure of the infinite 
iteration of i.i.d.\ Brownian motions. Thus
the random function $(L_a)_{a \in \mathbb{R}}$ must be thought of as 
the local time of this infinite iteration. 
The convergence of the last theorem will be obtained by proving 
convergence of finite dimensional marginals of the iterated 
Brownian motions, namely:
\begin{theorem} \label{thmfidi}
For any integer $p \geq 1$, there exists a random vector 
$(X_i)_{1 \leq i \leq p}$ such that for any pairwise distinct nonzero
real numbers $x_1, x_2, \ldots , x_p$ 
we have the following convergence in distribution 
\begin{eqnarray} \label{fidi} 
\big(W_n(x_i)\big)_{1 \leq i \leq p} & \xrightarrow[n\to\infty]{(d)} & (X_i)_{1 \leq i \leq p}.  
\end{eqnarray}
The random variables $X_1, \ldots, X_p$ and the differences
$X_i-X_j$, $i\neq j$, have all identical distribution which is
that of a ``signed exponential'' with parameter $2$, i.e.\ 
a distribution with density $e^{-2|x|}$, $x \in \mathbb R$.
\end{theorem}  
The case $p=1$
has already been noticed in the Physics literature \cite{Tur04} 
in the context of infinite iteration of i.i.d.\ random walks. 
We are unfortunately unable to give an explicit formula for
the distribution of $(X_1, \ldots, X_p)$
when $p \geq 2$, see Section \ref{comments} for a discussion 
and simulations of this intriguing probability distribution.
It is quite interesting to observe that the marginals and differences
have all the same distribution but are, of course, dependent.
In a certain sense, the infinite iteration is both self-similar at all scales
and long-range dependent.

The paper is organized as follows. The first section is devoted 
to the proof of Theorem \ref{thmfidi} which is the cornerstone 
in the proof of Theorem \ref{main}. We then turn to the study of 
occupation measures in Section 2. 
The last section presents some open questions and comments. \medskip 

\noindent \textbf{Acknowledgments:} We are deeply indebted to Yuval Peres 
for insightful discussions.

\section{Finite dimensional marginals}
In this section, we prove Theorem \ref{thmfidi}. 
The convergence of one-dimensional marginal is a special case because 
we  can explicitly give its limiting distribution, 
which will be of great use throughout this paper. 
In the general case, the convergence of finite dimensional marginals comes from ergodic property of random iterations of independent Brownian motions.

\subsection{One-dimensional marginals}
Let $\mathbb{R}^* := \mathbb{R} \setminus \{0\}$. For $\lambda >0$, we denote by  
$\pm \mathcal{E}(\lambda)$ a signed exponential distribution
with parameter $\lambda$, i.e., one which has density proportional to
$e^{-\lambda |x|}$, $x \in \mathbb R$.
\begin{proposition} \label{onedim} For any $t \in \mathbb{R}^*$ we have 
the following convergence in distribution 
\begin{eqnarray} \label{fidi1} 
W_n(t) & \xrightarrow[n\to\infty]{(d)} & \pm \mathcal{E}(2),  
\end{eqnarray}
\end{proposition}
\begin{remark} \label{nonconv} Proposition \ref{fidi1} 
already implies that the sequence of processes $(W_n)$ is not tight 
for the topology of uniform convergence on compact intervals. 
Indeed, if this were the case, and since $W_n(0)=0$,
we would have that
$\sup_{n \geq 1} P(W_n(\varepsilon)>\eta) \to 0$ as $ \varepsilon \to 0$, 
for every $\eta>0$, and this contradicts Proposition \ref{onedim}.
\end{remark}

\proof By standard properties of Gaussian variables we have  the following chain of equalities in distribution \begin{eqnarray*} {W}_n(t) &=& B_n\Big(...\big(B_3(B_2(B_1(t)))\big)...\Big) \\
&\overset{(d)}{=}&B_n\Big(...\big(B_3(B_2(\sqrt{|t|}B_1(1)))\big)...\Big)\\
& \overset{(d)}{=}& B_n\Big(...\big(B_3(|t|^{1/4}\sqrt{|B_1(1)|}B_2(1))\big)...\Big)\\
& \overset{(d)}{=}& \pm t^{2^{-n}} \prod_{i=0}^{n-1} |\mathcal{N}_i|^{2^{-i}},  \end{eqnarray*} where $ \mathcal{N}_{i}$ are i.i.d.\,standard normal variables and $\pm$ is a independent fair random sign.  It is then easy to see that the modulus of the right hand side of the last display actually converges almost surely as $n \to \infty$. Indeed the series of $P(|\log(| \mathcal{N}_i|^{2^{-i}})| > i^{-2})$ is summable and an application of the Borel-Cantelli lemma proves 
the claim. Thus, if we set 
\begin{eqnarray*} \mathcal{X} &:=& \lim_{n \to \infty} \prod_{i=0}^{n-1} |\mathcal{N}_i|^{2^{-i}} \ \ a.s., 
\end{eqnarray*} we have  the convergence $W_n(t) \to \pm \mathcal{X}$ in distribution as $n \to \infty$, where $\pm$ is fair random sign independent of $ \mathcal{X}$. Notice that the limit does not depend on $t$ as long as $t \ne 0$. To identify the limit distribution, 
we note that $ \mathcal{X}$ satisfies the following recursive distributional equation   
\[
\mathcal{X} \overset{(d)}{=} \sqrt{ \mathcal{X}} \cdot |\mathcal{N}|. 
\tag{E} 
\] 
Iterating this equation and applying the same arguments as above, it is easy to see that it admits a unique fixed point as long as 
$ \mathcal{X} >0$ almost surely. Guided by the result of \cite{Tur04}, an easy computation shows that the exponential variable of density $ 2e^{-2x} \mathbf{1}_{x>0}$ also satisfies (E) 
thus completing the proof of the proposition. 
\endproof

\subsection{General case}
The goal of this section is to prove Theorem \ref{thmfidi} 
in the case $p \geq 2$. The convergence  \eqref{fidi} will be 
achieved by applying arguments from the theory of Harris chains.
The idea is to consider the random transformation which associates 
with any $p$ points $x_1, \ldots , x_p \in \mathbb{R}$ 
the images $B(x_1), \ldots , B(x_p)$ of these $p$ points under a 
two-sided Brownian motion $B$ and to show that
independent applications of this map possess an ergodic property.
That is, for any initial state $(x_1,\ldots,x_p)$, the 
distribution of $(W_n(x_1), \ldots, W_n(x_p))$ converges weakly to
a unique invariant probability measure.

Let $p \geq 2$.  Denote by $ \mathcal{R}^p$ the set of $p$-uplets 
$(x_1, \ldots, x_p)$ of pairwise distinct nonzero real numbers.
Note that if $\mathbf{x}=( x_{1}, ... , x_{p}) \in \mathcal{R}^p$ 
then its image $B(\mathbf{x}) := (B(x_1), \ldots , B(x_p))$
under a two-sided Brownian motion $B$ almost surely 
belongs to $\mathcal{R}^p$.

\begin{figure}[!h]
 \begin{center}
 \includegraphics[width=10cm]{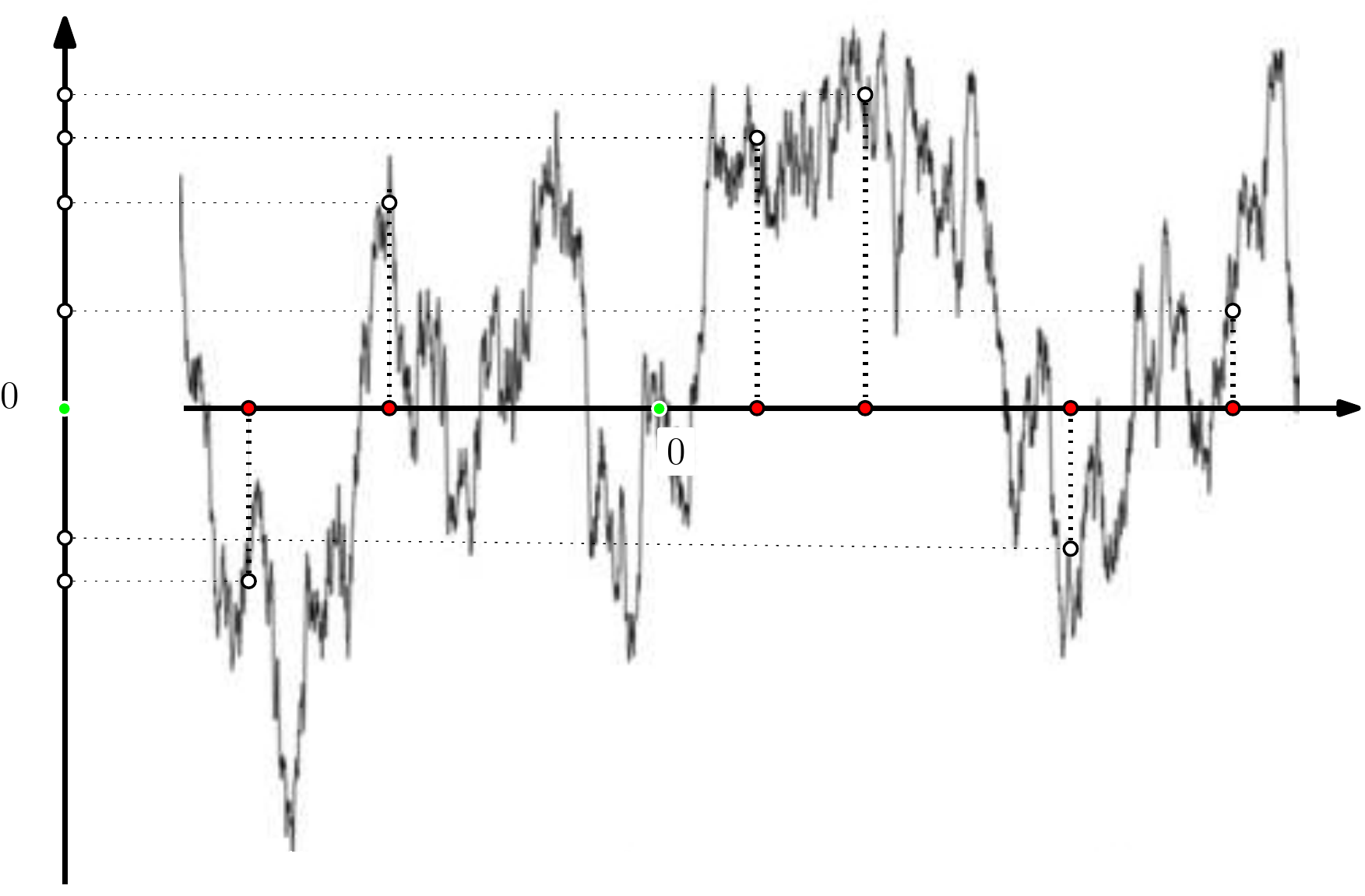}
 \caption{A $6$-uplet and its image after applying an independent Brownian motion.}
 \end{center}
 \end{figure}
 
\begin{proposition} \label{propfidi} 
For any $p \geq 2$, there is a unique probability measure $\nu_p$ on 
$\mathcal{R}^p$ such that if $(X_1, ... , X_p)$ is distributed according to 
$\nu_p$ and if $B$ is an independent two-sided Brownian motion then we have 
the equality in distribution
\begin{eqnarray} \label{stable} 
\big(B(X_i)\big)_{1 \leq i \leq p} &\eqdist& (X_i)_{1 \leq i \leq p}.  
\end{eqnarray}
Furthermore, for any $\mathbf{x}=(x_{1}, \ldots , x_{p}) \in \mathcal{R}^p$ 
we have the following convergence in distribution 
\[
\big( W_n(x_i)\big)_{1\leq i \leq p}  \xrightarrow[n\to\infty]{(d)} 
(X_i)_{1\leq i \leq  p}.  
\]
\end{proposition}
Notice that this argument does not give an explicit expression 
for the stationary probability measures $\nu_{p}$ but characterizes them 
uniquely by \eqref{stable}.  Let $p \geq 2$ be a fixed integer. 
Fix also a $p$-uplet $ \mathbf{x}=(x_{1}, \ldots , x_{p}) \in \mathcal{R}^p$. 
For every $n \geq 1$ we set 
\[
\mathcal{W}_{n} := (W_{n}(x_{1}), ... , W_{n}(x_{p})).  
\] 
Thus the process $ (\mathcal{W}_{n})_{n \geq 0}$ 
is a Markov chain with state space 
$ \mathcal{R}^p$, starting from $\mathcal{W}_0:= \mathbf{x}$,
with transition probability kernel given by 
\[
\mathsf P( \mathbf{y} ; A) = E\big[ (B(y_{1}), \ldots , B(y_{p})) \in A \big] , 
\]
for any $ \mathbf{y} \in \mathcal{R}^p$ and $A \subset \mathcal{R}^p$,
where $B$ is a two-sided Brownian motion. We denote by $E_{ \mathbf{x}}[\cdot]$ the expectation of this chain started from $ \mathbf{x} \in \mathcal{R}^p$. Since we restricted ourselves to $ \mathcal{R}^p$, the chain $( \mathcal{W}_{n})$ 
is easily seen to be irreducible with respect to the
$p$-dimensional Lebesgue measure on $ \mathcal{R}^p$ and aperiodic. 
We will show that the chain is in fact positive Harris recurrent, 
which will imply that it admits a unique invariant probability measure, 
and  Proposition \ref{propfidi} will directly flow from it, 
see \cite[Theorem 13.0.1]{MT09}. The key to prove this is to consider the following sets:

\begin{definition}
Fix $M>1$ a (large) real number. We say that the $p$-uplet $ \mathbf{x}=(x_{1}, ... , x_{p}) \in \mathcal{R}^p$ is $M$-sparse if we have 
$$  \displaystyle M^{-1} \le \sup_{1 \leq i \leq p}|x_i| \leq M,  \quad \mbox{ and }\quad 
 \displaystyle \min_{1 \leq i\ne j \leq p} |x_i-x_j| \geq M^{-1}.$$
 The set of all $M$-sparse $p$-uplets in $ \mathcal{R}^p$ is denoted by $ S_{M}$.
 \end{definition} The basic observation is that if $ \mathbf{x}=(x_1, ... , x_p)$ and $ \mathbf{y}=(y_1, ... , y_p)$ are two $M$-sparse $p$-uplets and if $B$ is a two-sided Brownian motion then the two random $p$-uplets $(B(x_1), \ldots , B(x_p))$ and $(B(y_1), \ldots , B(y_p))$ are mutually absolutely continuous with Radon-Nikod\'ym derivative bounded from below by a positive constant $c_{p,M}$ depending on $M$ and $p$ only. The reason is that the ratio of the two densities
$f(x_1, \ldots, x_p)/f(y_1, \ldots, y_p)$ is a continuous function on the compact set $S_M$. In other words, there exists two (dependent) Brownian motions $B$ and $\tilde{B}$ such that $B(x_1)=\tilde{B}(y_1), ... , B(x_p)= \tilde{B}(y_p)$ with probability at least $c_{p,M}$. Put
it otherwise, if we fix $ \mathbf{x} \in S_M$, then
\[
\mathsf P(\mathbf{y}, A) \ge c_{p,M} \mathsf P(\mathbf{y}, A),
\]
for all $\mathbf{y} \in S_M$ and all measurable $A \subset \mathcal R^p$.
Thus Ney's minorization condition holds and, therefore,
the set $ S_{M}$ is a petite set in the sense of \cite[Chapter 5]{MT09}. 
 
 \proof[Proof of Proposition \ref{propfidi}] 
In order to prove that $( \mathcal{W}_{n})$ is positive Harris recurrent, we will show that, for some $M > 0$, the expected return time to the petite set $S_{M}$, by the Markov chain
$(\mathcal W_n)$ started from $\mathbf x$, is bounded above by a finite constant, 
{\em uniformly} over all $\mathbf x \in S_M$.

The technical tool that we use here is the so-called drift condition, see \cite{MT09}, for
a Lyapunov or ``potential" function $V : \mathcal{R}^p \to  \mathbb{R}_{+}$. It is 
required that $V$ be unbounded on $\mathcal R^p$. The drift
of such a function is defined by
\[
DV(x) := E_{\mathbf x}[ V(\mathcal W_1)] - V(\mathbf x), \quad \mathbf x \in \mathcal R^p.
\]
Thus, $D$ is the generator of the Markov chain. We will show that there exists 
a Lyapunov function $V$ such that, for some $0 < a, b < \infty$,
\begin{eqnarray} \label{goal} DV(x) & \leq & -a + b \mathbf{1}_{ \mathbf{x} \in  S_{M}} , \end{eqnarray} for every $ \mathbf{x} \in \mathcal{R}^p$. 
More precisely, we show that the preceding condition is satisfied  for the function $V$ defined by
  \begin{eqnarray*} V(x_{1},\ldots , x_{p}) &:=& \max_{1 \leq i \leq p} |x_{i}| 
+ \sum_{0 \le i< j \le p} |x_{i}-x_{j}|^{-1/2},  \end{eqnarray*} and for some $a, b, M>0$. 
In the last display, we use the notation $x_0:=0$ to avoid extra terms in the definition of $V$.
It is convenient to consider the terms $U(\mathbf x) := \max_{1 \leq i \leq p} |x_{i}|$,
and $G(\mathbf x) := 
\sum_{0 \le i< j \le p} |x_{i}-x_{j}|^{-1/2}$
comprising $V$, separately. For any $\lambda > 0$,
\[
E_{\mathbf x} [U(\mathcal W_1)] 
= E\big[\max_{1 \le i \le p} |B(x_i)|\big] 
= \sqrt{\lambda} ~E\big[\max_{1 \le i \le p} |B(x_i/\lambda)|\big] 
\]
Letting $\lambda = U(\mathbf x)$, we thus have
\[
E_{\mathbf x} [U(\mathcal W_1)] \le C_1 \sqrt{U(\mathbf x)},
\]
where $C_1 = E\big[\max\{|B(t)| : -1 \leq t \leq 1\} \big]$.
For the other term, we have, with $ \mathcal{N}$ a standard normal variable,
\begin{multline*}
E_{\mathbf x} [G(\mathcal W_1)] 
= 
\sum_{0 \le i< j \le p} E\big[ |B(x_i) - B(x_j)|^{-1/2} \big]
= \sum_{0 \le i< j \le p}  |x_i-x_j|^{-1/4} E[ |\mathcal{N}|^{-1/2}]
\\ \le C_2 \sqrt{G(\mathbf x)},
\end{multline*}
where $C_2 = p E[ |\mathcal{N}|^{-1/2}]$.
Putting the terms together, we find
\[
E_{\mathbf x} [V(\mathcal W_1)] \le C_3 \sqrt{V(\mathbf x)},
\]
where $C_3 = 2 \max(C_1, C_2)$.
Thus, for all $x \in \mathcal R^p$,
\[
DV(\mathbf x) \le C_3 \sqrt{V(\mathbf x)} - V(\mathbf x).
\]
If $\mathbf x \in S_M$ then $V(\mathbf x) \le M + (p+1)^2 \sqrt{M} < \infty$.
If $\mathbf x \not \in S_M$ then $|x_i| > M$ for some $i$ or $|x_i-x_j| < 1/M$ for some
$0 \le i< j\le p$. In the first instance, $V(\mathbf x) > M$; in the second instance,
$V(\mathbf x) > \sqrt{M}$. So, for $M > 1$, we have $V(\mathbf x) > \sqrt{M}$, for 
all $\mathbf x \not \in S_M$. Let $M = \max(16 C_3^4, 1)$. Then for $\mathbf x \not \in S_M$,
we have $\sqrt{V(\mathbf x)} > 2 C_3$, implying that
$DV(\mathbf x) < - C_3 \sqrt{V(\mathbf x)} < - C_2 M^{1/4}$. Thus \eqref{goal} holds.
Theorem 13.0.1 in \cite{MT09} now applies, showing positive Harris recurrence. This finishes the proof of the proposition. 
\endproof 

\begin{remark}
The proof actually shows that there are constants $C_4, C_5 >0$ such that
\[
DV(\mathbf x) \le -C_4 \sqrt{V(\mathbf x)},
\]
when $V(\mathbf x)> C_5$, enabling one to establish a rate of convergence towards the
stationary distribution. We shall not pursue this further herein.
\end{remark}

\subsection{Properties of the $\nu_{p}$'s} 
We can think of $\nu_p$ as a measure on the product space $\mathbb R^p$
(by letting it have mass $0$ outside $\mathcal R^p)$.
The probability measures $(\nu_{p})_{p \geq 1}$ are consistent.
This follows from the fact that $\nu_p$ is the limit in distribution of
$(W_n(x_1), \ldots, W_n(x_p))$, whereas $\nu_{p+1}$
is the limit of $(W_n(x_1), \ldots, W_n(x_p), W_n(x_{p+1}))$, for 
any $(x_1, \ldots, x_p, x_{p+1}) \in \mathcal R^p$.
Therefore $\nu_p$ is the projection of $\nu_{p+1}$ on the first
$p$ coordinates. By Kolmogorov's extension theorem, there is a probability
measure $\nu$ on $\mathbb{R}^ \mathbb{N}$
such that $\nu_p$ is obtained from $\nu$ by projecting on
the first $p$ coordinates.
The family $((\nu_{p})_{p \geq 1}, \nu)$ has some further properties.

\subsubsection{Exchangeability}
The last statement of Proposition \ref{propfidi} actually shows
that $\nu$ is exchangeable, that is, invariant under permutations
of finitely many coordinates. 
Let $(X_1, X_2, \ldots)$ be a random element of $\mathbb R^{\mathbb N}$
with law $\nu$.
By the classical de Finetti/Ryll-Nardzewski/Hewitt-Savage theorem 
(Theorem 11.10 \cite{Kal02})
we know that there exists a random element of $\mathcal M$ 
(random probability measure) $\mu_\infty$, such that
the law of $(X_{1}, X_{2}, \ldots)$ conditionally on $\mu_\infty$
is a product measure (the law of i.i.d.\ random variables).
The random measure $\mu_\infty$ will be 
identified and interpreted in the next section as the occupation measure 
of the infinitely iterated Brownian motion. 
\subsubsection{Stationarity of the increments}
The family of measures $(\nu_{p})_{p \geq 1}$ also possesses another 
property which can be described as
stationarity of the increments. 
\begin{proposition} \label{stationarity}
Let $p \geq 2$ and let $(X_{1}, X_{2}, \ldots , X_{p})$ be distributed 
according to $\nu_{p}$. Then for every $1 \leq \ell \leq p$ we have 
\begin{eqnarray*}
(X_{1}-X_{\ell}, \ldots , X_{\ell-1}-X_{\ell},X_{\ell+1}-X_{\ell}, \ldots , X_{p}-X_{\ell}) & \overset{(d)}{=}& \nu_{p-1}.  
\end{eqnarray*}
\end{proposition}
\begin{proof}  Let $(x_{1}, x_{2}, \ldots , x_{p}) \in \mathcal{R}^p$. 
It is easy to prove by induction on $n \geq 1$ and using elementary 
manipulation of the Gaussian distribution that the random vector 
$(W_{n}(x_{i})-W_{n}(x_{\ell}))_{i\ne \ell }$ has the same distribution as 
the vector $(W_{n}(x_{i}-x_{\ell}))_{i \ne \ell}$. 
Notice that the vector $(x_{i}-x_{\ell})_{i \ne \ell} \in \mathcal{R}^{p-1}$, 
thus we can apply Proposition \ref{propfidi} and this finishes 
the proof of the proposition.
\end{proof}

\section{Occupation measure of the  infinitely iterated Brownian motion}
\subsection{Existence of $\mu_\infty$}
Recall the definition of the random measures $\mu_{n}$ given by formula 
\eqref{mun} in the Introduction, as well as the notion of convergence
in distribution of random elements of $(\mathcal M, \mathcal T)$.
With Theorem \ref{thmfidi} in our hands, it is now easy to prove convergence in distribution of the random measures $\mu_n$. 
We first need a lemma, characterizing convergence in distribution of random
elements of $(\mathcal M, \mathcal T)$, tailor-made for our case. The lemma
can be of independent interest.

\begin{lemma} \label{convergence}
Let $\lambda_1, \lambda_2, \ldots$ be random probability measures on 
$ \mathbb{R}$. The following are equivalent:
\begin{enumerate}
\item[(i)]
The sequence $\lambda_n$ converges in distribution to some
random probability measure.
\item[(ii)]
For each $n\ge 1$, and conditionally on $\lambda_n$, let
$X^n_1, X^n_2, \ldots$ be i.i.d.\ real-valued random variables each with
(conditional) law $\lambda_n$:
\[
P(X^n_1 \in A_1, \ldots, X^n_p \in A_p \mid \lambda_n) 
= \lambda_n(A_1) \cdots \lambda_n(A_p), \quad \text{ a.s.},
\]
for all $p \ge 1$, and Borel sets $A_1, \ldots, A_p$.
The random vector $(X^n_1, \dots, X^n_p)$ converges in distribution,
as $n \to \infty$, to some probability measure on $\mathbb R^p$.
%
\end{enumerate}
\end{lemma}
\begin{proof}
The implication (i) $\Rightarrow$ (ii) is an easy consequence of 
\cite[Theorem 4.2]{Kal76}.
For the other direction, fix $f \in C_K$.
We will show that the random variable $\lambda_n f 
= \int_{\mathbb R} f \mathrm{d} \lambda_n$
converges in distribution as $n \to \infty$.
Consider the random
probability measure   \begin{eqnarray*} \xi_n^{(p)} &:=& p^{-1} \sum_{i=1}^p \delta_{X^n_i}  \end{eqnarray*}
(i.e., the empirical distribution of $(X^n_1, \ldots, X^n_p)$).
We compare $\lambda_n f$ to $\xi_n^{(p)} f = p^{-1} \sum_{i=1}^p f(X_{i}^n)$.
We have
\[
P(|\lambda_{n}f - \xi_n f| \ge \varepsilon) 
= E\big[ P(|\lambda_{n}f - \xi_n f| \ge \varepsilon \mid \lambda_n) \big]
\le \frac{ \|f\|_{\infty}^2}{ \varepsilon^2 p},
\]
by Chebyshev's inequality.
On the other hand, $\xi_n^{(p)} f$ converges in distribution as $n \to \infty$.
Hence $\lambda_n f$ converges in distribution, and thus
$\lambda_n$ converges in distribution to some random measure $\lambda$.
 It remains to prove that $\lambda$ almost surely has mass one. Since $X_{1}^n$ converges in distribution, it follows that the sequence of random variables $\{X_{1}^n\}_{n \geq 1}$ is tight. Thus for every $\varepsilon>0$ there exists $M>0$ such that $P(|X_{n}^1|>M) \leq \varepsilon$ for every $n \geq 1$. Conditionally on $\lambda_{n}$ we have $P(|X_{1}^n|>M \mid \lambda_{n}) = 1-\lambda_{n}([-M,M])$. Taking expectation we deduce that  $E[\lambda_{n}([-M,M])] \geq 1-\varepsilon$ for every $n \geq 1$. This is sufficient to apply Theorem 4.9 in \cite{Kal86} and deduce that $\lambda$ is almost surely a random probability measure.
\end{proof}

Let us go back to our setting and show that the occupation measures $\mu_n$
converge to a random probability measure $\mu_\infty$. 
Let $p \geq 1$ and, conditionally on $\mu_{n}$, 
let $X_1^{n}, \ldots, X_p^{n}$ be i.i.d.\ random variables with
common distribution $\mu_n$. 
We show that $(X_1^{n}, \ldots, X_p^{n})$ converges (unconditionally) 
towards the random vector $(X_1, ..., X_p)$ of law $\nu_p$ 
identified in Proposition \ref{propfidi}. 
Indeed, by the definition of $\mu_n$, for any Borel bounded function 
$f: \mathbb{R}^p \to \mathbb{R}$ we have, by Fubini's theorem,
\begin{eqnarray*} 
E[ f(X_1^{(n)}, ...,X_p^{(n)})] 
&=& E\left[\int_0^1 \mathrm{d}u_1...\mathrm{d}u_p \  
f(W_n(u_1), ... , W_n(u_p))\right]\\
& =& \int_0^1 \mathrm{d}u_1...\mathrm{d}u_p \  
E\big[f(W_n(u_1), ... , W_n(u_p))\big].  
\end{eqnarray*} 
The last integral converges towards $E[f(X_1,...,X_p)]$ as $n\to \infty$ because of Theorem 	\ref{thmfidi} and dominated convergence. Applying Lemma \ref{convergence} we get the existence of a random probability measure $\mu_{ \infty}$ such that $\mu_n \to \mu_{ \infty}$ in distribution as $n \to \infty$.

\subsection{Support of $\mu_\infty$} 
\begin{proposition} \label{propsupport}
Almost surely, the random probability measure $\mu_\infty$ 
has a bounded support.
\end{proposition}

In order to prove the last proposition we use a very general fact:
\begin{lemma}  \label{support} Let $\lambda_n$ be a sequence of 
random probability measures converging in distribution towards 
a random probability measure $\lambda_\infty$. Suppose that the support 
of $\lambda_n$ is contained in $[A_n,B_n]$ and that the sequences of 
random variables $(A_n)$ and $(B_n)$ are tight. 
Then $\lambda_\infty$ has a compact support, almost surely. 
\end{lemma}
\proof Let us argue by contradiction and suppose that $\lambda_\infty$ 
has  probability at least $\varepsilon>0$ of having an unbounded support. 
By the assumption made on $A_n$ and $B_n$ there exists $M>0$ such that $
P(|A_n| \geq M) \leq \varepsilon/10$ and 
$P(|B_n| \geq M) \leq \varepsilon/10$ for every $n \geq 1$.  
For the $M>0$ chosen above, there exists $\delta >0$ such that  
we have $\lambda_{\infty}(]-{M,M}[^c)> \delta$ 
with probability at least $\varepsilon/2$. 
Now choose $p \geq \delta^{-1}$. 
Suppose that, conditionally on $\lambda_\infty$,
the random variables $X_1^{\infty}, \ldots, X_p^\infty$ are
i.i.d.\ with common law $\lambda_\infty$.  
We then have 
\begin{eqnarray} 
P\left(\sup_{1\leq i \leq p} |X_i^\infty| \geq M\right) 
& = & E\left[\left.P\left(\sup_{1\leq i \leq p} |X_i^\infty| \geq M \ \right| \
\lambda_{\infty}(]-M,M[^c)> \delta\right)\right]\nonumber \\ 
& \geq & \frac{\varepsilon}{2} \left(1- (1- \delta)^{p}\right)  
\quad \geq \quad \varepsilon (1-e^{-1})/2. 
\label{infini}  
\end{eqnarray}
On the other hand, if $X_1^n, \ldots ,X_p^n$ are i.i.d., 
conditionally on $\lambda_n$, we have
\begin{eqnarray} 
P \left( \sup_{1 \leq i \leq p}|X_i^n| \geq M  \right) 
& \leq & P(|A_n|\geq M) +P(|B_n|  \geq M)  \quad \leq \quad  \varepsilon/5. 
\label{fini}
\end{eqnarray} 
Since for $p$ fixed we have 
$(X_i^n)_{1 \leq i \leq p} \to (X_i^\infty)_{1 \leq i \leq p}$ in distribution 
as $n \to \infty$, 
comparing \eqref{infini} and \eqref{fini} leads to a contradiction. 
\endproof

We further establish a result on the limit of the oscillation of $W_n$ on
an interval, as $n \to \infty$. Recall, that the oscillation of
a function $f$ on an interval $J$ is defined by
$\osc(f; J) := \sup_{s,t \in J} |f(t)-f(s)|$,
and let, for $t > 0$,
\[
\Delta_n(t) := \osc(W_n; [0,t]).
\]
\begin{lemma}
\label{oscillation}
Let $D := \osc(B; [0,1])$ be the oscillation of a Brownian motion
on a unit  interval and let $D_1, D_2, \ldots$ be i.i.d.\ copies of $D$.
Then, for all $t > 0$,
 \begin{eqnarray*}
\Delta_n(t) &\xrightarrow[n \to \infty]{(d)}&
\prod_{i=0}^\infty D_i^{2^{-i}}.
 \end{eqnarray*}
\end{lemma}
\begin{proof}
Let
\[
I_n(t) := \inf_{x\in [0,t]} W_n(x), \quad
S_n(t) := \sup_{x\in [0,x]} W_n(x).
\]
Then $\Delta_n(t) = S_n(t) - I_n(t)$, and
\begin{align*}
\Delta_{n+1}(t) 
&=  \sup_{0 \le x \le t} B_{n+1}(W_n(x)) -
\inf_{0 \le x \le t} B_{n+1}(W_n(x)) 
\\
&= \sup_{I_n(t) \le u \le S_n(t)} B_{n+1}(u) 
- \inf_{I_n(t) \le u \le S_n(t)} B_{n+1}(u)
\\
&= \osc(B_{n+1}; [I_n(t), S_n(t)])
\\
&\eqdist \osc(B_{n+1}; [0, \Delta_n(t)])
\\
&= \sqrt{\Delta_n(t)} ~\osc(B_{n+1}; [0,1]).
\end{align*}
Thus iterating this equation, we get in the spirit of the proof 
of Proposition \ref{onedim} the following equality in distribution 
\[
\Delta_n(t) \overset{(d)}{=}  
t^{2^{-n}} \prod_{i=1}^{n} D_i^{2^{-(i-1)}}.
\]
An argument similar to the one used in the proof of Proposition \ref{onedim} 
shows that the right-hand side of the last display actually converges almost 
surely as $n \to \infty$, to the infinite product 
$\prod_{i=0}^{\infty} D_i^{2^{-i}}$. Hence $\Delta_n(t)$ converges in 
distribution to the same random variable.
\end{proof}

\begin{remark}
Roughly speaking,
the oscillation of the infinitely iterated Brownian motion is the same,
in distribution, over any interval of any length.
\end{remark}

\proof[Proof of Proposition \ref{propsupport}] 
By Lemma \ref{oscillation},
$\Delta_n(1)$ converges in distribution, it is tight. 
Thus the support of $\mu_n$ is contained in a compact interval
whose endpoints are tight.
Applying Lemma \ref{support}, we deduce that,
almost surely, $\mu_\infty$ has bounded support,
as required.
\endproof


\subsection{Density of $\mu_\infty$}  
This section is devoted to the analysis of the properties of the 
density of $\mu_\infty$. We will proceed in two steps. 
First, using standard technique of Fourier analysis for occupation densities  
(see e.g.\,\cite{Ber69}), we will prove that $\mu_\infty$ almost surely 
has a density which is in $ \mathbb{L}^{2}$. 
At the same time, we obtain some estimates about this Fourier transform. 
We will then use a very general result of Pitt \cite{Pit78} 
on local times to prove that this density is in fact continuous and even 
H\"older continuous with exponent $1/2 - \varepsilon$ 
for every $ \varepsilon>0$.

\subsubsection{Harmonic analysis of $\mu_\infty$}
For $\xi \in \mathbb{R}$, let 
\[
\Phi(\xi) 
:= \int_ \mathbb{R} \mathrm{d}\mu_\infty(x) \exp( \mathrm{i}\xi x),  
\]
be the Fourier transform of the random probability measure $\mu_\infty$.

\begin{proposition} 
\label{Fourier} 
For any $\xi \in \mathbb{R}$ we have   
\begin{equation} 
\label{fourier}  
E[|\Phi(\xi)|^2]= \frac{4}{4+\xi^2}.  
\end{equation}
\end{proposition}
\proof 

By the definition of the Fourier transform $\Phi(\xi)$, we have
\[
E\left[ |\Phi(\xi)|^2 \right] 
= E \left[  \int_ {\mathbb{R}^2}\mathrm{d}\mu_\infty(x) 
\mathrm{d}\mu_\infty(y) e^{i \xi (x-y)} \right].  
\]
By the convergences already established we have 
\[
E \left[\int_ {\mathbb{R}^2}\mathrm{d}\mu_\infty(x) 
\mathrm{d}\mu_\infty(y) e^{i \xi (x-y)} \right] 
= \lim_{n \to \infty} E\Big[ \exp\big(i\xi(W_n(U_1)-W_n(U_2))\big)\Big], 
\]
where $U_1, U_2$ are two independent random variables uniformly 
distributed over $[0,1]$ and also independent of the sequence of
Brownian motions $B_1, B_2, \ldots$. By stationarity of the increments 
(see the proof of Proposition \ref{stationarity}), 
for any $s,t \in \mathbb{R}$, we have $W_n(t)-W_n(s) = W_n(t-s)$ in distribution,
and thus, 
\begin{eqnarray*} 
E \left[\int_ {\mathbb{R}^2}\mathrm{d}\mu_\infty(x) 
\mathrm{d}\mu_\infty(y) e^{i \xi (x-y)} \right] 
&=& \lim_{n \to \infty} E\Big[ \exp\big(i\xi(W_n(U_1-U_2))\big)\Big]\\ 
&=& \lim_{n \to \infty} \iint_0^1 \mathrm{d}u \mathrm{d}v \, 
E \left[\exp(i\xi W_n(u-v))\right].  
\end{eqnarray*}
Applying Proposition \ref{fidi1} and the dominated convergence theorem 
we get that 
\begin{eqnarray*}
E \left[\int_ {\mathbb{R}^2}\mathrm{d}\mu_\infty(x) 
\mathrm{d}\mu_\infty(y) e^{i \xi (x-y)} \right] 
\quad=\quad E[\exp( i \xi [\pm \mathcal{E}(2)])]  
\quad = \quad \frac{4}{4+\xi^2}. 
\end{eqnarray*}
\endproof

In particular, applying Fubini's theorem we deduce from \eqref{fourier} 
that 
\[
E[\|\Phi\|_2^2] = \int_{\mathbb R}
\left(\frac{4}{4+\xi^2}\right)^2 \mathrm{d}\xi
< \infty,
\]
and thus that $\Phi \in \mathbb{L}^2$ 
almost surely. By standard results on Fourier transforms,
this implies that, almost surely,  $\mu_\infty$ has density 
$(L_a)_{a\in \mathbb{R}}$ with respect to the Lebesgue measure
which is in $ \mathbb{L}^2$. Notice that the estimates 
\eqref{fourier} in fact give us a bit more than that, 
namely, for every $ 0<s <1/2 $ we have 
\begin{eqnarray*} 
\|\Phi\|_{H^s} 
:= \sqrt{\int_ \mathbb{R} \mathrm{d}\xi \ (1+\xi^2)^s |\Phi(\xi)|^2} 
&<& \infty,  \qquad \mbox{a.s.} 
\end{eqnarray*}
Applying the standard Sobolev inequality, we deduce that for every $0<s<1/2$ we have
\[
\|L\|_{ \mathbb{L}^{2/(1-2s)}}  \quad \leq  \quad C \|\Phi\|_{H^s},
\]
where $C$ is a universal constant;
see Theorem 1.38 in \cite{BCD11}. Since $s$ can be arbitrarily close to $1/2$ and since $L$ is itself a density (thus $L \in \mathbb{L}^1$ a.s.) we get that  $L \in \mathbb{L}^q$  
for any $1 \leq q < \infty$ almost surely.

\subsubsection{Continuity of the density} 
The idea to prove continuity of the density $(L_a)_{a \in \mathbb{R}}$ 
is to apply once again a Brownian motion ``on top of $\mu_{\infty}$'' 
and to use the general theory developed in \cite{Pit78}. 
Formally, if $\mu_\infty$ has density $(L_a)_{a\in \mathbb{R}}$ 
(which is in every $ \mathbb{L}^q, 1 \leq q < \infty$) 
we define the random measure $\widetilde{\mu}_\infty$ by 
 \begin{eqnarray*} 
\widetilde{\mu}_\infty f 
&:=&\int \mathrm{d}{\mu}_\infty(x) \ f(B(x)) 
= \int \mathrm{d}a\,L_a \ f(B(a)), 
\end{eqnarray*}
where $B$ is an independent two-sided Brownian motion. 
Clearly, we have 
\[
\widetilde{\mu}_\infty \eqdist \mu_\infty
\]
(equality in distribution). 
So it suffices to show that  $\widetilde{\mu}_{\infty}$ has a 
continuous density. This follows from the lemma:
\begin{lemma} 
Let $g : \mathbb{R} \to \mathbb{R}_+$ be the density of 
a probability measure supported on a compact interval 
$I \subset \mathbb{R}$ such that $g \in \mathbb{L}^q$ for every 
$1 \leq q < \infty$. Let also $B$ be a two-sided Brownian motion 
and consider the occupation measure of $B$ with respect 
to the measure $g(t)\mathrm{d}t$ defined as follows
\begin{eqnarray*} 
\mu f&=& \int_ \mathbb{R} g(t) \mathrm{d}t \,f( B(t)),  
\end{eqnarray*} 
for any Borel positive function $f$. 
Then the random measure $\mu$ almost surely has a density which is locally Hölder continuous of exponent $1/2-\varepsilon$ for any 
$ \varepsilon>0$.
\end{lemma}
\proof 
Let $ 0<\delta <1$ and choose $q \geq 1$ such that $(1+\delta) q^* <2$ 
where $q^*$ is defined by $1/q+ 1/q^* =1$. 
Applying H\"older's inequality we get that  
\begin{eqnarray*} 
\sup_{s \in I} \int _I \mathrm{d}t \,\frac{g(t) }{\sqrt{|t-s|}^{1+\delta}} 
& \leq & \left( \int_I \mathrm{d}t \, |g(t)|^q\right)^{1/q} 
\sup_{s \in I} \left(	\int_I \mathrm{d}t\, 
|t-s|^{-(1+\delta)q^*/2}\right)^{1/q^*} \quad <\quad  \infty.  
\end{eqnarray*}
We can thus apply Theorem 4 of \cite{Pit78} and get that $\mu$ a 
continuous density $L$ over $ \mathbb{R}$ which satisfies a 
local Hölder condition $|L(x)-L(y)|  \leq  K |x-y|^{\delta'}$,   
for any $\delta' < \delta/2$.
\endproof

\section{Comments} \label{comments}
\paragraph{Occupation over other measures.} In this work we considered 
the occupations measures $\mu_{n}$ of $W_{n}$ 
over the time interval $[0,1]$ but the proofs apply in the more 
general setting of occupation measures of $W_{n}$ over any 
probability measure on $ \mathbb{R}$ which is not atomic. More precisely if $\lambda$ is a probability measure on $ \mathbb{R}$ which has no atom, we define the occupation of $W_n$ over $\lambda$ by 
 \begin{eqnarray*} \mu_n^{(\lambda)}f &:=& \int_ \mathbb{R} \lambda( \mathrm{d}t) f(W_n(t)),  \end{eqnarray*} for any Borel $f : \mathbb{R} \to \mathbb{R}_+$. Then the random probability measures $\mu_n^{(\lambda)}$ converge towards $\mu_\infty$ in distribution as well.
 
\paragraph{Explicit finite dimensional marginals.} 
The consistent family distributions $\{\nu_{p}, p \geq 1\}$ 
introduced in Proposition 
\ref{propfidi} are the limiting distributions of the finite-dimensional 
marginals of the $W_{n}$'s. They are characterized by \eqref{stable}. 
Although they arise naturally in the study of iteration of Brownian motions, 
to the best of our knowledge, they have not been investigated so far for 
$p \geq 2$ and in particular no explicit formulae are known 
for the density of $\nu_{p}$, $p \geq 2$.

\begin{figure}[!h]
\begin{center}
\includegraphics[width=9cm]{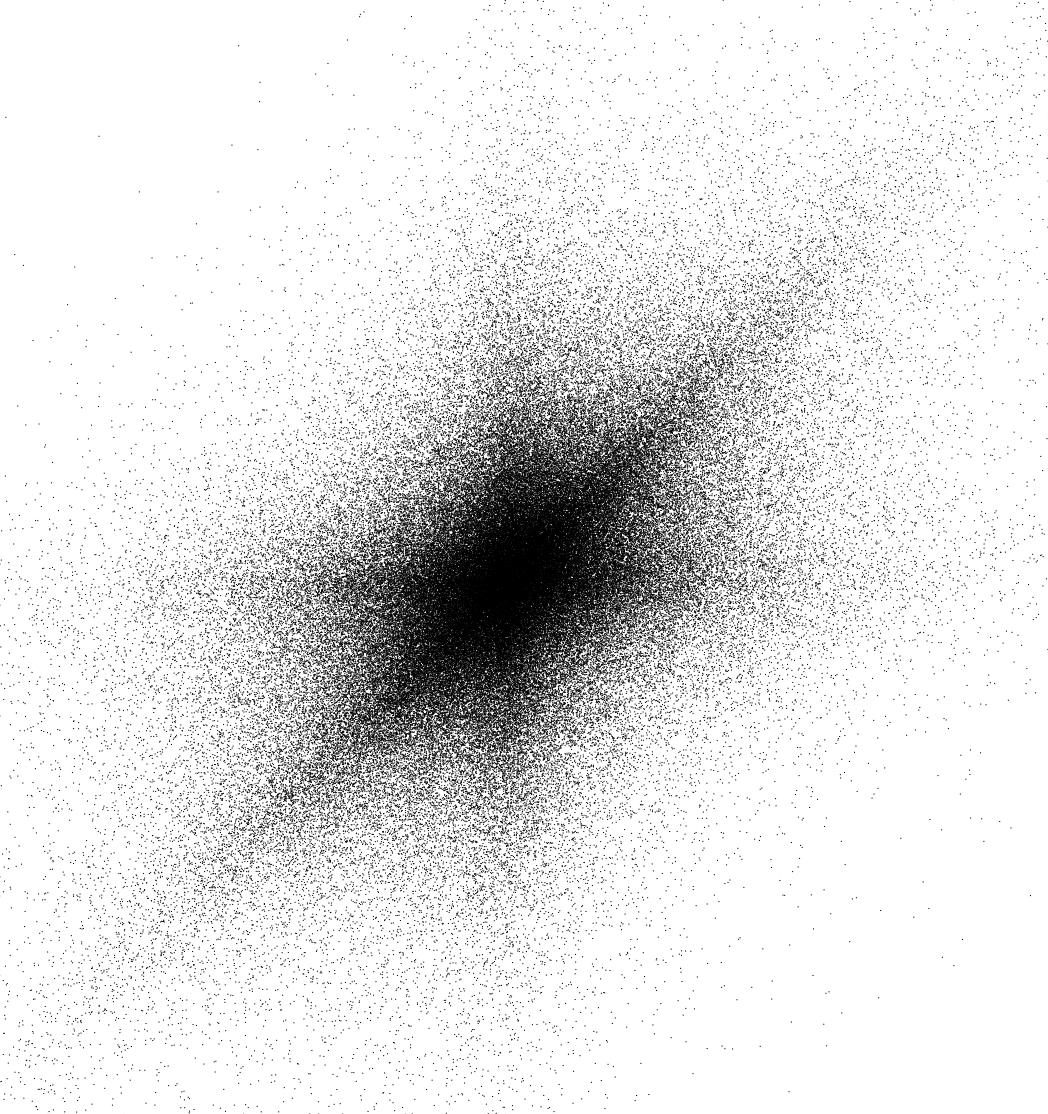}
\begin{minipage}{10cm}
\caption{ \label{seastar}
About $300 000$ independent samples of 
the distribution $\nu_{2}$. Simulation realized by Alexander Holroyd. 
}
\end{minipage}
\end{center}
\end{figure} 
In Fig.\,\ref{seastar} many independent points have been sampled according 
to $\nu_{2}$. A clear shape ``sea star'' emerges and we conjecture that the 
level-lines of the density of $\nu_{2}$  would be dilatation of this unique 
shape. However we do not have any candidate for this density.
 
\paragraph{Funaki's PDE.}
Following Funaki \cite{Fun79}, we can show that
$W_n$ solves the PDE $\partial u/\partial t = \partial^k u/\partial x^k$,
where $k=2^n$, with initial data $u(t=0,x)=u_0(x)$, $x \in \mathbb R$, 
where $u_0$
is a function satisfying growth conditions \cite{Fun79}.
The solution is given by
$u_n(t,x) = E[u_0(x+ W_n(2^{2^n-1} t))]$.
Although it is probably not possible to give a PDE meaning to
the exchangeable collection of random variables
$(W_\infty(x))_{x \in \mathbb R}$, we remark that the
limit to the signed exponential translates into
$\lim_{n \to \infty} u_n(t/2^{2^n-1}, x) =
\int_{\mathbb R} u_0(x-y) e^{-2|y|} dy$.

\paragraph{Ray-Knight theorem.} 
In the spirit of the famous  Ray-Knight theorems, do we have another way
to describe the density of $\mu_{\infty}$ as a diffusion process? 

\paragraph{Reflected Brownian motion.} 
It is possible to extend part of the previous work to iterations of 
reflected Brownian motions. Namely, it is likely that Proposition 
\ref{propfidi} goes through but Proposition \ref{stationarity} 
fails, and thus the analysis of the Fourier transform of the 
limiting measure would be more challenging.

\paragraph{Fractional Brownian motion.} 
Fractional Brownian motion generalizes Brownian motion in that it
is a Gaussian $H$-self-similar process with stationary increments,
where $0 < H < 1$.  Again, it is likely that Proposition \ref{propfidi} go through, but the limit
is even harder to describe. 
Of course, this raises the following interesting question:
what kind of processes can be iterated {\em ad libitum} and
result into some kind of limit?

\bibliographystyle{alpha}

\bigskip

\noindent
\begin{tabular}{lcl}
Nicolas Curien & \hspace{1cm} & Takis Konstantopoulos\\
D\'epartement de Math\'ematiques et Applications  &&  Department of Mathematics  \\
\'Ecole Normale Sup\'erieure, 45 rue d'Ulm & & Uppsala University P.O. Box 480\\
75230 Paris cedex 05, France && 751 06 Uppsala, Sweden \\ \ \\

nicolas.curien@ens.fr && takis@math.uu.se

\end{tabular}
 \end{document}